\RequirePackage{ifpdf} 
\ifpdf
\documentclass[11pt,pdftex]{article}
\else
\documentclass[11pt,dvips]{article}
\fi

\ifpdf
\fi
\usepackage[bookmarks, 
colorlinks=false, backref,plainpages=false]{hyperref}
\usepackage[english]{babel}
\usepackage{graphics}
\usepackage{epsfig}
\usepackage{amsmath,amssymb,amsthm}

\renewcommand{\theequation}{\thesection.\arabic{equation}}

\newtheorem{thm}{Theorem}[section]

\newtheorem{lem}[thm]{Lemma}

\newtheorem{rem}[thm]{Remark}

\begin{document}
\newcommand{\BX}{{\bf X}}
\newcommand{\cv}{{\cal V}}
\newcommand{\cW}{{\cal W}}
\newcommand{\co}{{\cal O}}

\renewcommand{\theequation}{\thesection.\arabic{equation}}
\def\@eqnnum{{\reset@font\rm (\theequation)}}

\newtheorem{theorem}{Theorem}[section]
\newtheorem{corollary}{Corollary}[section]
\newtheorem{lemma}{Lemma}[section]
\newtheorem{proposition}{Proposition}[section]
\newtheorem{conjecture}{Conjecture}[section]
\newtheorem{remark}{Remark}[section]
\newtheorem{definition}{Definition}[section]
\newtheorem{problem}{Problem}[section]

\def\abstract{
\advance \rightskip by 10mm
\advance \leftskip by 10mm
\vspace{-0.8em}
\noindent
\small{\bf Abstract.}
}
\def\endabstract{\par\normalsize\rm}

\def\Xint#1{\mathchoice
{\XXint\displaystyle\textstyle{#1}}%
{\XXint\textstyle\scriptstyle{#1}}%
{\XXint\scriptstyle\scriptscriptstyle{#1}}%
{\XXint\scriptscriptstyle\scriptscriptstyle{#1}}%
\!\int}
\def\XXint#1#2#3{{\setbox0=\hbox{$#1{#2#3}{\int}$}
\vcenter{\hbox{$#2#3$}}\kern-.5\wd0}}
\def\ddashint{\Xint=}
\def\dashint{\Xint-}

\def\a{\alpha}
\def\b{\beta}
\def\d{\delta}\def\D{\Delta}
\def\e{\epsilon}
\def\g{\gamma}\def\G{\Gamma}
\def\k{\kappa}
\def\lam{\lambda}\def\Lam{\Lambda}
\renewcommand\o{\omega}\renewcommand\O{\Omega}
\def\s{\sigma}\def\S{\Sigma}
\renewcommand\t{\theta}\def\vt{\vartheta}
\newcommand{\vphi}{\varphi}
\def\z{\zeta}

\newcommand{\tsigma}{\tilde{\s}}
\newcommand{\tbsigma}{\tilde{\bsigma}}
\def\te{\tilde{\e}}
\def\tu{\tilde{u}}

\newcommand{\bchi}{\mbox{\boldmath$\chi$}}
\newcommand{\bdelta}{\mbox{\boldmath$\delta$}}
\newcommand{\bepsilon}{\mbox{\boldmath$\epsilon$}}
\newcommand{\bfeta}{\mbox{\boldmath$\eta$}}
\newcommand{\bgamma}{\mbox{\boldmath$\gamma$}}
\newcommand{\bomega}{\mbox{\boldmath$\omega$}}
\newcommand{\bvphi}{\mbox{\boldmath$\varphi$}}
\newcommand{\bphi}{\mbox{\boldmath$\phi$}}
\newcommand{\bPhi}{\mbox{\boldmath$\Phi$}}
\newcommand{\bpsi}{\mbox{\boldmath$\psi$}}
\newcommand{\bPsi}{\mbox{\boldmath$\Psi$}}
\newcommand{\bsigma}{\mbox{\boldmath$\sigma$}}
\newcommand{\btau}{\mbox{\boldmath$\tau$}}
\newcommand{\bxi}{\mbox{\boldmath$\xi$}}
\newcommand{\brho}{\mbox{\boldmath$\rho$}}

\def\bk{\boldsymbol{\kappa}}
\def\bmu{\boldsymbol\mu}
\def\bxi{\boldsymbol{\xi}}
\def\bz{\boldsymbol{\zeta}}

\def\ba{{\bf a}}
\def\bb{{\bf b}}
\def\bc{{\bf c}}
\def\be{{\bf e}}
\def\bff{{\bf f}}
\def\bg{{\bf g}}
\def\bn{{\bf n}}
\def\bp{{\bf p}}
\def\bq{{\bf q}}
\def\bs{{\bf s}}
\def\bt{{\bf t}}
\def\bu{{\bf u}}
\def\bv{{\bf v}}
\def\bw{{\bf w}}
\def\bx{{\bf x}}
\def\by{{\bf y}}
\def\bzz{{\bf z}}

\def\bD{{\bf D}}
\def\bE{{\bf E}}
\def\bF{{\bf F}}
\def\bH{{\bf H}}
\def\bJ{{\bf J}}
\def\bV{{\bf V}}
\def\bU{{\bf U}}
\def\bW{{\bf W}}
\def\bX{{\bf X}}
\def\bY{{\bf Y}}

\def\cA{{\cal A}}
\def\cC{{\cal C}}
\def\cD{{\cal D}}
\def\cE{{\cal E}}
\def\cF{{\cal F}}
\def\cG{{\cal G}}
\def\cI{{\cal I}}
\def\cJ{{\cal J}}
\def\cK{{\cal K}}
\def\cL{{\cal L}}
\def\cO{{\cal O}}
\def\cP{{\cal P}}
\def\cQ{{\cal Q}}
\def\cR{{\cal R}}
\def\cS{{\cal \Sigma}}
\def\cT{{\cal T}}
\def\cU{{\cal U}}
\def\cV{{\cal V}}

\def\scT{{_\cT}}
\def\sD{{_D}}
\def\sE{{_E}}
\def\sF{{_F}}
\def\sFz{{_{F_z}}}
\def\sK{{_K}}
\def\sI{{_I}}
\def\sb{{_b}}
\def\sN{{_N}}

\def\curl{{{\bf curl} \ }}
\def\rot{{\mbox{rot}\ }}
\def\BPI{{\bf \Pi}}

\def\cth{\cT_h}
\def\ctH{\cT_H}

\def\tJ{\tilde{\J}}

\def\hK{\widehat{K}}
\def\hx{\widehat{x}}
\def\hy{\widehat{y}}
\def\bhv{\widehat{\bv}}

\def\l{\ell}
\def\bl{\boldsymbol{\ell}}
\def\col{\colon}
\def\f12{\frac12}
\def\dfrac{\displaystyle\frac}
\def\dint{\displaystyle\int}
\def\nab{\nabla}
\def\p{\partial}
\def\sm{\setminus}
\def\dsum{\displaystyle\sum}
\newcommand{\pp}[2]{\frac{\partial {#1}}{\partial {#2}}}
\def\bzero{{\bf 0}}

\def\divv{\nab\cdot}
\def\divx{\nab_x\cdot}
\def\divtx{\nab_{t,x}\cdot}
\def\nabx{\nab_x}

\newcommand{\grad}{\nabla}
\newcommand{\curlt}{{\nabla \times}}
\newcommand{\gperp}{\nabla^{\perp}}
\newcommand{\gradt}{\nabla\cdot}

\def\forallqq{\quad\forall\,}
\def\aph{A^{1/2}}
\def\amh{A^{-1/2}}

\def\osc{{\rm osc \, }}

\def\Im{{\rm Im}}
\newcommand{\tr}{{\rm tr}}
\def\divvr{{\rm div}}
\def\curllr{{\rm curl}}
\def\curll{{\rm curl}}
\def\curl{{\bf curl}}
\newcommand{\bgrad}{{\bf grad}}
\newcommand\diam{\mathrm{diam\,}}
\renewcommand\Im{\mathrm{Im\,}}
\def\Span{\mbox{Span}}
\def\supp{\mbox{supp\,}}
\newcommand{\trace}{{\rm trace}}

\newcommand{\tri}{|\!|\!|}
\newcommand{\ljump}{\lbrack\!\lbrack}
\newcommand{\rjump}{\rbrack\!\rbrack}
\newcommand{\bdm}{\begin{displaymath}}
\newcommand{\edm}{\end{displaymath}}
\newcommand{\beq}{\begin{equation}}
\newcommand{\eeq}{\end{equation}}
\newcommand{\beqa}{\begin{eqnarray}}
\newcommand{\eeqa}{\end{eqnarray}}
\newcommand{\beqas}{\begin{eqnarray*}}
\newcommand{\eeqas}{\end{eqnarray*}}
\newcommand{\ul}{\underline}
\newcommand{\wh}{\widehat}
\newcommand{\la}{\langle}
\newcommand{\ra}{\rangle}

\newcommand{\Lt}{L^2(\Omega)}
\newcommand{\Lts}{L^2(\Omega)^2}
\newcommand{\Ltc}{L^2(\Omega)^3}
\newcommand{\Ho}{H^1(\Omega)}
\newcommand{\Hoh}{H^1(\wh{\Omega})}
\newcommand{\Hoi}{H^1(\Omega_i)}
\newcommand{\Hos}{H^1(\Omega)^2}
\newcommand{\Hoc}{H^1(\Omega)^3}
\newcommand{\Hoch}{H^1(\wh{\Omega})^3}
\newcommand{\Hoci}{H^1(\Omega_i)^3}
\newcommand{\Hoz}{H^1_0(\Omega)}
\newcommand{\Ht}{H^2(\Omega)}
\newcommand{\Hti}{H^2(\Omega_i)}
\newcommand{\Hts}{H^2(\Omega)^2}
\newcommand{\Htc}{H^2(\Omega)^3}
\newcommand{\Htz}{H^0(\Omega)}
\newcommand{\Hh}{H^{1/2}(\Gamma)}
\newcommand{\Hhi}{H^{1/2}(\Gamma_i)}
\newcommand{\Hmh}{H^{-1/2}(\Gamma)}
\newcommand{\Hdiv}{H(\divvr;\,\Omega)}
\newcommand{\Hdivh}{H(\divv;\,\wh \Omega)}
\newcommand{\hcurl}{H(\curl\,A;\,\Omega)}
\newcommand{\Hcurl}{H(\curll\,A;\,\Omega)}
\newcommand{\Hcrl}{H(\curll\,;\,\Omega)}
\newcommand{\hcrl}{H(\curl\,;\,\Omega)}
\newcommand{\Hcrlh}{H(\curll\,;\,\wh\Omega)}
\newcommand{\hcrlh}{H(\curl\,;\,\wh\Omega)}
\newcommand{\Wdiv}{\BW_0(\mbox{\divv}\,;\,\Omega)}
\newcommand{\Wcurl}{\BW_0(\mbox{\curl}\,A;\,\Omega)}
\newcommand{\WcrossV}{\BW \times V}

\def\calS{{\cal S}}
\def\calT{{\cal T}}
\def\cB{{\cal B}}
\def\cH{{\cal H}}
\def\ba{{\mathbf{a}}}
\def\cM{{\cal M}}
\def\cN{{\cal N}}

\def\bE{{\bf E}}
\def\bS{{\bf S}}
\def\br{{\bf r}}
\def\bW{{\bf W}}
\def\bLambda{{\bf \Lambda}}

\newcommand{\lJump}{[\![}
\newcommand{\rJump}{]\!]}
\newcommand{\jump}[1]{[\![ #1]\!]}

\newcommand{\sd}{\bsigma^{\Delta}}
\newcommand{\st}{\tilde{\bsigma}}
\newcommand{\sh}{\hat{\bsigma}}
\newcommand{\rd}{\brho^{\Delta}}

\newcommand{\WH}{W\!H}
\newcommand{\NE}{N\!E}

\newcommand{\ND}{N\!D}
\newcommand{\BDM}{B\!D\!M}

\newcommand{\sT}{{_T}}
\newcommand{\sRT}{{_{RT}}}
\newcommand{\sBDM}{{_{BDM}}}
\newcommand{\sWH}{{_{WH}}}
\newcommand{\sND}{{_{ND}}}
\newcommand{\sV}{_\cV}
\newcommand{\sB}{_B}

\newcommand{\dd}{\underline{{\mathbf d}}}
\newcommand{\C}{\rm I\kern-.5emC}
\newcommand{\R}{\rm I\kern-.19emR}
\newcommand{\W}{{\mathbf W}}
\def\3bar{{|\hspace{-.02in}|\hspace{-.02in}|}}
\newcommand{\A}{{\mathcal A}}

\title{Div First-Order System LL* (FOSLL*)  
for Second-Order Elliptic Partial Differential Equations
\thanks{
This work performed under the auspices of the U.S. Department of Energy by 
Lawrence Livermore National Laboratory under Contract DE-AC52-07NA27344 (LLNL-JRNL-645325). 
This work was supported in part by the National Science Foundation
under grant DMS-1217081.}}
\author{Zhiqiang Cai\thanks{
Department of Mathematics, Purdue University, 150 N. University
Street, West Lafayette, IN 47907-2067, caiz@purdue.edu.}
\and Rob Falgout\thanks{Center for Applied Scientific Computing, Lawrence Livermore 
National Laboratory, Livermore, CA 94551-0808, falgout2@llnl.gov.}
\and Shun Zhang\thanks{Department of Mathematics, City University of Hong Kong, Hong Kong, shun.zhang@cityu.edu.hk.}}

\maketitle

\begin{abstract}
The first-order system LL* (FOSLL*) approach for general second-order elliptic partial 
differential equations was proposed and analyzed in \cite{CaMaMcRu:01}, 
in order to retain the full efficiency of 
the $L^2$ norm first-order system least-squares (FOSLS) approach while exhibiting the 
generality of the inverse-norm FOSLS approach.  
The FOSLL* approach in \cite{CaMaMcRu:01} was applied to the div-curl system with added slack variables, 
and hence it is quite complicated. 
In this paper, we apply the FOSLL* approach to the div system and establish its well-posedness.
For the corresponding finite element approximation, we obtain a quasi-optimal a priori error bound 
under the same regularity assumption as the standard Galerkin method, but without the 
restriction to sufficiently small mesh size.
Unlike the FOSLS approach, the FOSLL* approach does not have a free a posteriori error estimator, we then propose an explicit residual
error estimator and establish its reliability and efficiency bounds.
\end{abstract}

\smallskip

{\bf Key words.} LL* method, least-squares method, a priori error estimate, a posteriori error estimate, elliptic
equations.

\smallskip

{\bf AMS(MOS) subject classifications. 65M60, 65M15}
\pagestyle{myheadings}
\thispagestyle{plain}
\markboth{Z. Cai, R. Falgout and S. Zhang}{Div FOSLL* for 2nd-oder elliptic PDEs}

\section{Introduction}\label{sec-introduction}
\setcounter{equation}{0}

There are substantial interests in the use of least squares
principles for the approximate solution of partial differential
equations with applications in both solid and fluid mechanics. 
Many least-squares methods for the scalar elliptic partial differential equations
have been proposed and
analyzed, \cite{BocGun:09, Jia:98}. Their numerical properties depend on choices such as the
first-order system and the least-squares norm. Loosely speaking,
there are three types of least-squares methods: the inverse
approach, the div approach, and the div-curl approach. The inverse
approach employs an inverse norm that is further replaced by
either the weighted mesh-dependent norm (see \cite{aks}) or the
discrete $H^{-1}$ norm (see \cite{blp}) for computational
feasibility. Both the div and the div-curl approaches use the $L^2$ norm 
and the
corresponding homogeneous least-squares functionals are equivalent
to the $H(\divvr)\times H^1$ and the
$\left(H(\divvr)\cap H(\curllr)\right) \times H^1$ norms, respectively. 
The div approach based on the flux-pressure
formulation has been studied by many researchers (see, e.g.,
\cite{BoGu:05, clmm, Jes:77, PeCaLa:94}). The div-curl approach
\cite{CaMaMc:97} has also been well studied.

In order to retain the full efficiency of 
the $L^2$ norm first-order system least-squares (FOSLS) approach while exhibiting the 
generality of the inverse-norm FOSLS approach, the first-order system LL* (FOSLL*) approach for general second-order elliptic partial 
differential equations was proposed and analyzed in \cite{CaMaMcRu:01}.
 The FOSLL* approach 
 was applied to the div-curl system, whose adjoint system is an underdetermined system and hence
 is not suitable for FOSLL*. This difficulty was overcome by carefully adding 
 slack variables to the div-curl system. But the resulting approach is quite complicated.
 
Our purpose here is to study the FOSLL* approach applying to the div system. Without adding 
any slack variables to the div system, the resulting approach is much simpler than that in \cite{CaMaMcRu:01}.
By showing that the bilinear form of the FOSLL* approach is coercive and bounded
and that the linear form is bounded with respect to a weighted $H(\divvr)\times H^1$ norm,
we establish the well-posedness of the FOSLL* approach. Under the same regularity assumption as the standard Galerkin method, but without the restriction to sufficiently small mesh size, we obtain 
a quasi-optimal a priori error bound for the corresponding finite element approximation.
Note that this assumption is weaker than that for the div FOSLS \cite{CaKu:10}.
Unlike the FOSLS approach, the FOSLL* approach does not have a free a posteriori error estimator, 
thus we study an explicit residual
error estimator and establish its reliability and efficiency bounds.



The paper is organized as follows. In Section~2 we introduce
mathematical equations for the second-order scalar elliptic
partial differential equations and its div first-order system, and 
we then derive the FOSLL* variational formulation and 
establish its well-posedness.  In Section~3,
the FOSLL* finite element approximation 
is described. A priori and a posteriori error estimations are obtained in Sections~4 and 5 respectively. 
 In Section 6, we present numerical results.

\subsection{Notation}
\setcounter{equation}{0}

We use the standard notations and definitions for the Sobolev
spaces $H^s(\O)^d$ and $H^s(\partial\O)^d$ for $s\ge 0$. The
standard associated inner products are denoted by $(\cdot , \,
\cdot)_{s,\O}$ and $(\cdot , \, \cdot)_{s,\partial\O}$, and their
respective norms are denoted by $\|\cdot \|_{s,\O}$ and
$\|\cdot\|_{s,\partial\O}$. (We suppress the superscript $d$
because the dependence on dimension will be clear by context. We
also omit the subscript $\O$ from the inner product and norm
designation when there is no risk of confusion.) For $s=0$,
$H^s(\O)^d$ coincides with $L^2(\O)^d$. In this case, the inner
product and norm will be denoted by $\|\cdot\|$ and
$(\cdot,\,\cdot)$, respectively.  Set
 \[
 H^1_D(\Omega):=\{q\in H^1(\Omega)\, :\, q=0\,\,\mbox{on}\,\,\Gamma_D\}.
 \]
When $\Gamma=\partial\O$, denote $H^1_D(\Omega)$ by
$H^1_0(\Omega)$. Finally, set
\[
H(\divvr;\O)=\{\bv\in L^2(\O)^d\, :\,\gradt\bv\in L^2(\O)\},
\]
which is a Hilbert space under the norm
\[
\|\bv\|_{H(\divvr;\,\O)}=\left(\|\bv\|^2+\|\gradt\bv\|^2
 \right)^{1/2},
\]
and define the subspace
\[
H_N(\divvr;\O)=\{\bv\in H(\divvr;\O)\, :\,\bn\cdot\bv=0
\,\,\mbox{on}\,\,\Gamma_N\}.
\]

\section{First-Order System LL* Formulation}\label{sec3}
\setcounter{equation}{0}

Let $\Omega$ be a bounded, open, connected subset of $\Re^d$ ($d=2$
or $3$) with a Lipschitz continuous boundary $\partial\Omega$.
Denote by $\bn=(n_1,\, ...\, ,\,n_d)$ the outward unit vector normal to
the boundary. We partition the boundary of the domain $\Omega$ into
two open subsets $\Gamma_D$ and $\Gamma_N$ such that
$\partial\Omega=\bar{\Gamma}_D\cup\bar{\Gamma}_N$ and $\Gamma_D\cap
\Gamma_N=\emptyset$. For simplicity, we assume that $\Gamma_D$ is
not empty (i.e., $\mbox{mes}\,(\Gamma_D)\not= 0$) and is connected. 

\subsection{Second-Order Elliptic Problem}

Consider the following second-order elliptic boundary value
problem:
 \begin{equation}\label{scalar}
 -\nabla\cdot \,(A\nabla\, u ) + \bb\cdot \grad\, u+ a\,u =f
 \quad \mbox{in} \,\,\Omega
 \end{equation}
with boundary conditions
 \beq\label{bc1}
 u =g_{_D}\quad \mbox{on} \,\,\Gamma_D
 \quad\mbox{and}\quad
- A \nabla\, u\cdot\bn =g_{_N}\quad \mbox{on} \,\,\Gamma_N,
 \eeq
where the symbols $\nabla\cdot$ and $\nabla$ stand for the
divergence and gradient operators, respectively; $A$ is a given
$d\times d$ tensor-valued function; $\bb\in L^\infty(\Omega)^d$ and $a\in L^\infty(\Omega)$ 
are given vector- and scalar-valued
functions, respectively; and $f$ is a given scalar function. Assume
that $A$ is uniformly symmetric positive definite: there exist
positive constants $0 < \Lambda_0 \leq \Lambda_1$ such that
 \[
 \Lambda_0 \bxi^T\bxi \leq \bxi^TA \bxi
 \leq \Lambda_1 \bxi^T\bxi
 \]
for all $\bxi \in \Re^d$ and almost all $x \in \overline{\Omega}$.
The corresponding variational form of system
(\ref{scalar})-(\ref{bc1}) is to find $u\in H^1(\O)$ such that $u|_{\Gamma_D}=g_\sD$ and that
 \beq\label{s-galerkin}
 a(u,\,v)=(f,\,v)-\int_{\Gamma_N} g_\sD v\,ds \quad\forall\,\,v\in H^1_D(\O),
 \eeq
where the bilinear form is defined by
 \[
 a(u,\,v)=(A\grad\, u,\,\grad\, v)+(\bb\cdot \grad\, u+ a\,u,\,v).
 \]
The dual problem of (\ref{s-galerkin}) is to find $z\in H^1(\O)$
such that $z|_{\Gamma_D}=g_\sD$ and that
 \beq\label{s-dual}
 \hat{a}(\phi,\,z)= (f,\,v) -\int_{\Gamma_N}g_\sN v\,ds 
 \quad\forall\,\,\phi\in H^1_D(\O),
 \eeq
 where the bilinear form is defined by
 \[
 \hat{a}(\phi,\,z)=\left(\grad\,\phi,\, (A\,\grad +\bb)\,z\right)
 +(\phi,\,a\,z).
 \]
Assume that both problems
(\ref{s-galerkin}) and (\ref{s-dual}) have unique solutions and, 
for simplicity of the presentation, 
satisfy the full $H^2$
regularity estimates:
 \beq\label{reg-s}
 \|u\|_2\leq C\,\|f\|
 \quad\mbox{and}\quad
 \|z\|_2\leq C\,\|f\|.
 \eeq
Here and
thereafter, we use $C$ with or without subscripts in this paper to
denote a generic positive constant, possibly different at
different occurrences, that is independent of the mesh size $h$
but may
depend on the domain $\O$.

\subsection{First-Order System}
Introducing the flux (vector) variable
 \[
 \bsigma=-A\,\nabla u ,
 \]
the scalar elliptic problem in (\ref{scalar})-(\ref{bc1}) may be rewritten as
the following first-order partial differential system:
 \begin{equation}\label{fos-s}
 \left\{\begin{array}{lcclc}
 A^{-1}\bsigma+\nabla u & = & \bzero & \quad \mbox{in} \, &\Omega
 ,\\[2mm]
 \gradt\bsigma - \bb\cdot A^{-1}\bsigma +a\, u& = & f &\quad \mbox{in} \, &\Omega
 \end{array}
 \right.
 \end{equation}
with boundary conditions
 \begin{equation}\label{bc-s}
 u =g_{_D} \quad \mbox{on} \,\, \Gamma_D \quad\mbox{and}\quad
 \bn \cdot \bsigma= g_{_N}\quad \mbox{on} \,\, \Gamma_N.
 \end{equation}
 
 Let 
 \[
{\cal L} = \left(\begin{array}{cc}
A^{-1} & \grad \\[2mm]
\gradt -\bb\cdot A^{-1} &  a
\end{array}\right),
\quad 
{\cal U}=\left(\begin{array}{c}
\bsigma\\[2mm]
u
\end{array}\right),
\quad\mbox{and}\quad
{\cal F}=\left(\begin{array}{c}
\bzero\\[2mm]
f
\end{array}\right),
\]
then (\ref{fos-s}) may be rewritten as
  \begin{equation}\label{fos-s2}
 {\cal L}\, {\cal U}= {\cal F}.
 \end{equation}
 
 \subsection{Div FOSLL* Variational Formulation}
 
 Multiplying test function ${\cal V}=(\btau,\,v)^t\in H_N(\divvr;\O)\times H^1_D(\O)$,
 integrating over the domain $\Omega$, and using integration by parts, we have
\begin{eqnarray*}
 (f,\,v)
 &=& \langle  {\cal F},\, {\cal V} \rangle
 =\langle {\cal L}\, {\cal U},\, {\cal V} \rangle\\[2mm]
 &=& (A^{-1}\bsigma+\nabla u,\, \btau)
      + (\gradt\bsigma -\bb\cdot A^{-1}\bsigma + a\,u,\,v)\\[2mm]
  &=& (\bsigma,\,A^{-1}\btau)-(\bsigma,\,\grad\,v)
              +\int_{\partial\Omega}(\bsigma\cdot\bn)\,v\,ds
              - (\bsigma,\,A^{-1}\bb\,v)\\[2mm]
      && -(u,\,\gradt\,\btau) + \int_{\partial\Omega}(\btau\cdot\bn)\,u\,ds
           + (u,\,a\, v)\\[2mm]
   &=& \big(\bsigma,\, A^{-1}\btau - (\grad + A^{-1}\bb)\,v\big) 
     + \big(u,\, a\,v -\gradt\btau \big)
       + \int_{\Gamma_N} g_\sN v\,ds + \int_{\Gamma_D} g_\sD (\btau\cdot\bn)\,ds\\[2mm]
   &=& \langle  {\cal U},\,  {\cal L}^*\,{\cal V} \rangle + g(\btau,\,v),
\end{eqnarray*}
where the formal adjoint of ${\cal L}$ and the boundary functional are defined by
 \[
  {\cal L}^* = \left(\begin{array}{cc}
A^{-1} & -(\grad +A^{-1}\bb) \\[2mm]
-\gradt  &  a
\end{array}\right)
\quad\mbox{and}\quad
g(\btau,\,v)= \int_{\Gamma_N} g_\sN v\,ds + \int_{\Gamma_D} g_\sD (\btau\cdot\bn)\,ds,
\]
respectively.

Without loss of generality, we assume that $a\not= 0$ in this paper and let 
 \[
 \cA = \left(\begin{array}{cc}
A &  0\\[2mm]
0 &  a^{-1}
\end{array}\right) .
 \]
Let ${\cal W}=(\bfeta,\,w)^t$ satisfy 
 \beq\label{w}
 {\cal U}= 
\cA\, {\cal L}^*\,{\cal W}
=\left(\begin{array}{c}
\bfeta - (A\grad+\bb) \, w\\[2mm]
-a^{-1}\gradt  \bfeta + w
\end{array}\right), 
\eeq
then we have
 \[
  \langle \cA\, {\cal L}^*\,{\cal W},\,  {\cal L}^*\,{\cal V} \rangle  = (f,\,v) - g(\btau,\,v)
\equiv  f(\btau,\,v).
 \]
Now, our div FOSLL* variational formulation is to find $(\bfeta,\,w)\in 
H_N(\divvr;\O)\times H^1_D(\O)$ such that
\beq\label{fosll*}
 b(\bfeta,\,w;\,\btau,\,v)=f(\btau,\,v),
 \quad\forall\,\, (\btau,\,v)\in H_N(\divvr;\O)\times H^1_D(\O),
 \eeq
 where the bilinear form $b(\cdot,\,\cdot)$ is defined by
  \begin{eqnarray*}
   && b(\bfeta,\,w;\,\btau,\,v)
 =\langle \cA\, {\cal L}^*\,{\cal W},\,  {\cal L}^*\,{\cal V} \rangle \\[2mm]
  &=& \left(\bfeta -(A\,\grad +\bb)\, w,\,A^{-1}\btau - (\grad+A^{-1}\bb)\,v \right) 
   + (a^{-1}\gradt\,\bfeta - w,\, \gradt\,\btau-a\, v).
  \end{eqnarray*}
Note that both non-homogenous Dirichlet and Neumann boundary conditions are imposed weekly. 

\begin{rem}
For any $(\bfeta,\,w),\,\, (\btau,\,v)\in H_N(\divvr;\O)\times H^1_D(\O)$, integration by parts gives
 \[
 (\grad w,\btau)+(w, \gradt \btau) = (\grad v,\bfeta)+(v, \gradt \bfeta)= 0. 
 \]
Hence, the bilinear form $b(\cdot,\,\cdot)$ has of the form
 \begin{eqnarray*}
  b(\bfeta,\,w;\,\btau,\,v)
  &\!\!=\!\!& (A^{-1}\bfeta,\,\btau)+(a^{-1}\gradt\bfeta,\,\gradt\btau)
         +(A\grad w,\, \grad v) +(a\, w,\,v) \\[2mm]
  &&\!\!\! - (\bb\, w, \,A^{-1}\btau) -  
   (A^{-1}\bfeta,\, \bb\, v) + (\bb\, w,\, \grad v) + (\grad w,\, \bb\, v) 
 + (A^{-1}\bb\, w,\, \bb\, v) .
  \end{eqnarray*}
In the case that $\bb = \bzero$ and $a\not= 0$, i.e., the diffusion-reaction problem, the div FOSLL* problem  
in {\em (\ref{fosll*})} is decoupled. More specifically, $w\in H^1_D(\O)$ is the solution of
 \[
 (A\grad w,\, \grad v) +(a\, w,\,v)=(f,\,v)-\int_{\Gamma_N} g_{_N} v\,ds,
 \quad\forall\,\, v\in H^1_D(\O),
\]
and $\bfeta\in H_N(\divvr;\O)$ satisfies 
 \[
   (A^{-1}\bfeta,\,\btau)+(a^{-1}\gradt\bfeta,\,\gradt\btau)
 =-\int_{\Gamma_D} g_{_D} (\btau\cdot\bn)\,ds.
 \]
Note that the problem for $w$ is similar to the standard variational formulation for the 
diffusion-reaction problem, but the non-homogeneous Dirichlet boundary condition is weakly imposed here.
\end{rem}

\begin{rem}
In the case that $a= 0$,  let $\cA=\mbox{diag}\,(A,\,1)$,
then ${\cal W}=(\bfeta,\,w)^t$ satisfy 
 \beq\label{w}
 {\cal U}= \cA\, {\cal L}^* {\cal W}
= \left(\begin{array}{c}
\bfeta  -(A\grad+\bb) \,w\\[2mm]
-\gradt \bfeta
\end{array}\right).
\eeq
The corresponding bilinear form is modified as follows
$$
  b(\bfeta,\,w;\,\btau,\,v)=
   \big(\bfeta -(A\,\grad +\bb)\, w,\,A^{-1}\btau - (\grad+A^{-1}\bb)\,v \big) 
  + \big(\gradt\,\bfeta,\, \gradt\,\btau\big).
$$
\end{rem}

 \subsection{Well-Posedness}
 
 Denote by
 \[
 \tri v\tri_1 = \left(\|a^{1/2}v\|^2+\|A^{1/2}\grad\,v\|^2\right)^{1/2}
 \mbox{ and }\,\,
 \tri \btau\tri_{H(\divvr)} = \left(\|A^{-1/2}\btau\|^2+\|a^{-1/2}\gradt\btau\|^2 \right)^{1/2}
    \]
the weighted $H^1(\Omega)$ and $H(\divvr;\O)$ norms, respectively. Let
 \[
 \tri(\btau,\, v)\tri=\left(\tri v\tri^2_1 + \tri \btau\tri^2_{H(\divvr)} \right)^{1/2}.
 \]
The following theorem establishes the coercivity and continuity of the bilinear form.

\begin{thm}\label{co:equi}
 The bilinear form $b(\cdot,\,\cdot)$ is coercive and continuous in 
 $H_N(\divvr;\O)\times H^1_D(\O)$, i.e., there exist positive constant $\alpha$ and $C$,
 depending on bounds of the coefficients ($A$, $\bb$, and $a$),
 such that 
  \begin{equation}\label{coercivity}
  \alpha\,  \tri(\btau,\,\bv)\tri^2
  \leq b(\btau,\,v;\,\btau,\,v)
  \end{equation}
 and that 
  \begin{equation}\label{continuity}
 b(\bfeta,\,w;\,\btau,\,v)
 \leq C\,  \tri(\bfeta,\, w)\tri\, \tri(\btau,\, v)\tri
  \end{equation}
 for all $(\bfeta,\, u), \,\,(\btau,\,v)\in H_N(\divvr;\O)\times H^1_D(\O)$.
\end{thm}

A similar result to that of Theorem~\ref{co:equi} was proved in \cite{clmm}. 
For the convenience of readers, we provide
a comprehensive proof here.

\begin{proof}
(\ref{continuity}) is a direct consequence of the
Cauchy-Schwarz and the triangle inequalities and the bounds of the coefficients ($A$, $\bb$, and $a$) of the underlying problem. To show the validity of
(\ref{coercivity}), we first establish that
 \beq\label{s-ell-div-1}
   \tri(\btau,\, v)\tri^2 \leq C\,\left(b(\btau,\,v;\,\btau,\,v)+\|v\|^2\right)
 \eeq
for all $(\btau,\,v)\in H_N(\divvr;\O)\times H^1_D(\O)$.
 To this end, integrating by parts gives
   \[
  (\btau,\,\grad\,v)=(-a^{-1/2}\gradt\btau + a^{1/2}v,\,a^{1/2}v)-(a\,v,\,v).
  \]
It then follows from the Cauchy-Schwarz inequality that 
 \beqas
 && \|A^{1/2}\grad\,v\|^2 +\|a^{1/2}v\|^2\\[2mm]
 &=& \left(A^{1/2}\grad\,v-A^{-1/2}(\btau-\bb\,v),\,A^{1/2}\grad\,v\right)
    + \left(a^{1/2}v -a^{-1/2}\gradt\btau ,\,a^{1/2}v\right)
     - \left(\bb\,v,\,\grad\,v\right)\\[2mm]
 &\leq & \left( \|A^{1/2}\grad\,v-A^{-1/2}(\btau- \bb\,v)\|+ C\, \|v\|\right) 
 \|A^{1/2}\grad\,v\|
    +\|a^{-1/2}\gradt\btau - a^{1/2}v\|\,\|a^{1/2}v\|,
 \eeqas
which 
implies
 \beq\label{s-tem1}
 \|A^{1/2}\grad\,v\|^2 +\|a^{1/2}v\|^2
 \leq C\,\left(b(\btau,\,v;\,\btau,\,v)+\|v\|^2\right).
 \eeq
By the triangle inequality and (\ref{s-tem1}), we have that
 \beqas
  \|A^{-1/2}\btau\|
  &\leq &\left(\|A^{-1/2}(\btau-\bb\,v)-A^{1/2}\grad\,v\|
  +\|A^{1/2}\grad\,v\| + C\, \|v\|\right)\\[2mm]
 & \leq & C\,\left(b(\btau,\,v;\,\btau,\,v)^{1/2}+\|v\|\right).
 \eeqas
 and that
 \[
 \|a^{-1/2}\gradt\btau\|
 \leq  \|a^{-1/2}\gradt\btau -a^{1/2}v\| +\|a^{1/2}v\|
  \leq C\, \left(b(\btau,\,v;\,\btau,\,v)+\|v\|^2\right).
 \]
Combining the above three inequalities yields (\ref{s-ell-div-1}).

With (\ref{s-ell-div-1}), we now show the validity of
(\ref{coercivity}) by the standard compactness argument. To this end, assume
that (\ref{coercivity}) is not true. This implies that there exists
a sequence $\{\btau_n,\,v_n\}\in H_N(\divvr;\,\O)\times H_D^1(\Omega)$ such
that
 \beq\label{assum}
 \tri \btau_n\tri^2_{H(\divvr)} +\tri v_n\tri^2_1 =1
 \quad\mbox{and}\quad
 b(\btau_n,\,v_n;\,\btau_n,\,v_n)\leq \dfrac{1}{n}.
 \eeq
Since $H^1_D(\Omega)$ is compactly contained in $L^2(\O)$, there exists a
subsequence $\{v_{n_k}\}\in H^1_D(\Omega)$ which converges in $L^2(\O)$. For
any $k,\,\,l$ and $(\btau_{n_k},\,v_{n_k})$,
$(\btau_{n_l},\,v_{n_l})\in H_N(\divvr;\,\O)\times H^1_D(\Omega)$, it follows
from (\ref{s-ell-div-1}) and the triangle inequality that
 \beqas
 &&\tri \btau_{n_k}-\btau_{n_l}\tri^2_{H({\rm div})}
  +\tri v_{n_k}-v_{n_l}\tri^2_{1,\,\O}\\[2mm]
  &\leq & C\left(b(\btau_{n_k}-\btau_{n_l},\,v_{n_k}-v_{n_l};\, 
  \btau_{n_k}-\btau_{n_l},\,v_{n_k}-v_{n_l})
   +\|v_{n_k}-v_{n_l}\|^2\right)\\[2mm]
 &\leq & C\left(b(\btau_{n_k},\,v_{n_k};\,\btau_{n_k},\,v_{n_k})
 +b(\btau_{n_l},\,v_{n_l};\,\btau_{n_l},\,v_{n_l})
   +\|v_{n_k}-v_{n_l}\|^2\right)\\[2mm]
  & \to & 0,
   \eeqas
as $k,\,\,l \to \infty$. This  implies that $(\btau_{n_k},\,v_{n_k})$ is a Cauchy sequence
in the complete space $H_N(\divvr;\,\O)\times H^1_D(\Omega)$. Hence, there
exists $(\btau,\,v)\in H_N(\divvr;\,\O)\times H^1_D(\Omega)$ such that
 \[
 \lim\limits_{k\to\infty}\left(\tri \btau_{n_k}-\btau\tri_{H({\rm div})}
 +\tri v_{n_k}-v\tri_{1}\right)=0.
 \]

Next, we show that
 \beq\label{s-0}
 v=0\quad\mbox{and}\quad \btau=\bzero,
 \eeq
which contradict with (\ref{assum}) that
 \[
 0=\tri \btau\tri^2_{H({\rm div})}+\tri v\tri^2_{1}
 =\lim\limits_{k\to\infty}
 \left( \tri \btau_{n_k}\tri^2_{H({\rm div})}+\tri v_{n_k}\tri^2_{1}\right) =1.
 \]
To this end, for any $\phi\in H^1_D(\Omega)$, integration by parts, the
Cauchy-Schwarz inequality, and (\ref{assum}) give
 \beqas
 \hat{a}(\phi,\,v_{n_k})
 &=&\left(\grad\,\phi,\, (A\,\grad + \bb)\, v_{n_k}\right)+(\phi,\,a\,v_{n_k})\\[2mm]
 &=& \left(\grad \phi,\, (A\,\grad +\bb)\,v_{n_k} - \btau_{n_k}\right)
 +(\phi,\, a\,v_{n_k}-\gradt\btau_{n_k})\\[2mm]
 &\leq & b(\btau_{n_k},\,v_{n_k};\,\btau_{n_k},\,v_{n_k})^{1/2}\tri \phi\tri_{1}
\leq \left(\dfrac{1}{n_k}\right)^{1/2} \tri \phi\tri_{1}.
 \eeqas
Since $\lim_{k\to\infty} v_{n_k}=v$ in $H^1(\O)$, we then have
 \[
 |\hat{a}(\phi,\, v)|=\lim\limits_{k\to\infty}|\hat{a}(\phi,\,v_{n_k})|
 \leq \lim\limits_{k\to\infty}
 \left(\dfrac{1}{n_k}\right)^{1/2}\tri \phi\tri_{1}=0.
 \]
Because (\ref{s-dual}) has a unique solution, we have that
 \[
 v=0.
 \]
Now, $\btau=\bzero$ follows from (\ref{s-ell-div-1}):
 \[
 \tri \btau\tri^2_{H({\rm div})}
 =\lim\limits_{k\to\infty}\tri \btau_{n_k}\tri^2_{H({\rm div})}
 \leq C\,\lim\limits_{k\to\infty}\left(
 b(\btau_{n_k},\,v_{n_k};\, \btau_{n_k},\,v_{n_k})
   +\|v_{n_k}\|^2\right)=0.
   \]
This completes the proof of (\ref{s-0}) and, hence, the theorem.
\end{proof}


\begin{thm}\label{wellposedness}
 The variational formulation in {\em (\ref{fosll*})} has a uniques solution  
 $(\bfeta,\,w)\in H_N(\divvr;\O)\times H^1_D(\O)$ satisfying the following a priori estimate
  \begin{equation}\label{apriori}
  \tri (\bfeta,\,w)\tri \leq C\,\left(\|f\|_{-1,\Omega}+\|g_\sD\|_{1/2,\Gamma_D}
    +\|g_\sN\|_{-1/2,\Gamma_N}\right).
    \end{equation}
 \end{thm}

\begin{proof}
For all $(\btau,\,v)\in H_\sN(\divvr;\O)\times H^1_D(\O)$,
it follows from the definition of the dual norms and the trace theorem that 
 \begin{eqnarray*}
 |f(\btau,\,v)|
 &\leq &\|f\|_{-1,\Omega}\|v\|_{1,\Omega} 
 + \|g_\sD\|_{1/2,\Gamma_D}\|\btau\cdot\bn\|_{-1/2,\Gamma_D}
 +\|g_\sN\|_{-1/2,\Gamma_N}\|v\|_{1/2,\Gamma_D}\\[2mm]
  &\leq & C\,\left(\|f\|_{-1,\Omega}+\|g_\sD\|_{1/2,\Gamma_D}
    +\|g_\sN\|_{-1/2,\Gamma_N}\right) \,\tri (\btau,\,v)\tri.
\end{eqnarray*}
Now,  by the Lax-Milgram lemma, the well
possedness of (\ref{fosll*}) 
and the a priori estimate in (\ref{apriori}) follow directly from Theorem~\ref{co:equi}.
\end{proof}

\begin{rem}
In the case that $a=0$,  the norms are modified as follows
 \[
 \tri v\tri_1 = \left(\|A^{1/2}\grad\,v\|^2\right)^{1/2}
 \quad\mbox{and}\quad
 \tri \btau\tri_{H(\divvr)} = \left(\|A^{-1/2}\btau\|^2+\|\gradt\btau\|^2 \right)^{1/2}.
  \]
 With these norms, all the results obtained in this paper for $a\neq 0$ hold.
\end{rem}

\section{Div FOSLL* Finite Element Approximation}\label{sec4}
\setcounter{equation}{0}

Theorem~\ref{co:equi} guarantees that confirming finite element
spaces of $H_N(\divvr;\O)\times H^1_D(\O)$ for the vector and scalar variables,
$\bfeta$ and $w$, may
be chosen independently. 
However, the only finite element spaces
having optimal approximations in terms of both the regularity
and the approximation property are the
continuous piecewise polynomials for the scalar variable and the
Raviart-Thomas (or Brezzi-Douglas-Marini) elements for the vector variable. 
(The BDM element has slight more degrees of freedom than that of the RT element.)
Moreover, the
system of algebraic equations resulting from these elements can be
solved efficiently by fast multigrid methods. For the above
reasons, only these elements are analyzed in this paper. But it is
easy to see that our analysis does apply to any other conforming
finite element spaces with no essential modifications.

For simplicity of presentation, we consider only triangular and
tetrahedra elements for the respective two and three dimensions.
Assuming that the domain $\O$ is polygonal, let $\cth$ be a
regular triangulation of $\O$ (see \cite{Cia:78}) with
triangular/tetrahedra elements of size $\co (h)$. Let $P_k(K)$ be
the space of polynomials of degree $k$ on triangle $K$ and denote
the local Raviart-Thomas space of order $k$ on $K$:
 \[
  RT_k(K)=P_k(K)^d +\bx\,P_k(K)
  \]
with $\bx=(x_1,\, ...,\,x_d)$. Then the standard $\Hdiv$
conforming Raviart-Thomas space of index $k$ \cite{rt} and the
standard (conforming) continuous piecewise polynomials of degree
$k+1$ are defined, respectively, by
\begin{eqnarray*}\label{S_h}
\Sigma^k_h&=&\{\btau\in H_N(\divvr; \O)\,:\,
 \btau|_K\in RT_k(K)\,\,\,\,\forall\,\,K\in\cth\},\\[2mm]
 \label{V_h} \qquad\quad
\mbox{and }\quad V^{k+1}_h&=&\{ v\in H_D^1(\O)\,:\, v|_K\in
 P_{k+1}(K)\,\,\,\,\forall\,\,K\in\cth\}
 \end{eqnarray*}
It is well-known (see \cite{Cia:78}) that $V^{k+1}_h$ has the
following approximation property: let $k\ge 0$ be an integer and
let $l\in [0,\,k+1]$
 \begin{equation}\label{app2}
  \inf_{v\in V^{k+1}_h}\|u -v\|_1
  \leq C\,h^l\,\|u\|_{l+1},
  \end{equation}
for $u\in H^{l+1}(\O)\cap H^1_D(\Omega)$. It is also well-known (see
\cite{rt}) that $\Sigma^k_h$ has the following approximation
property: let $k\ge 0$ be an integer and let $l\in [1,\,k+1]$
\begin{equation}\label{app1}
   \inf_{\btau\in\,\Sigma^k_h}\|\bsigma -\btau\|_{\Hdiv}
  \leq C\,h^l\left(\|\bsigma\|_{l}
  +\|\gradt\bsigma\|_l\right)
  \end{equation}
for $\bsigma\in H^{l}(\O)^{d} \cap H_N(\divvr;\Omega)$ with
$\gradt\bsigma \in H^{l}(\O)^m$.
Since $\bsigma$ and $\gradt\bsigma$ are one order less smooth than
$u$, we will choose $k$ to be the smallest integer greater than
or equal to $l-1$.

The finite element discretization of the FOSLL* variational
problem is: find $(\bfeta_h,\,w_h)\in \Sigma^k_h\times V^{k+1}_h$
such that
\begin{equation}\label{fem}
 b(\bfeta_h,\,w_h;\,\btau,\,v)
 =f(\btau,\,v),\quad \forall\,\,
 (\btau,\,v)\in \Sigma^k_h\times V^{k+1}_h.
 \end{equation}

Since $\Sigma^k_h\times V^{k+1}_h$ is a
subspace of $H_N(\divvr;\O)\times H_D^1(\O)$, the div FOSLL* problem in (\ref{fem}) is well-posed and the solution continuously depends on the data.

\begin{thm}\label{wellposedness-h}
 The variational formulation in {\em (\ref{fem})} has a uniques solution  
 $(\bfeta_h,\,w_h)\in \Sigma^k_h\times V^{k+1}_h$ satisfying the following a priori estimate
  \begin{equation}\label{apriori}
  \tri (\bfeta_h,\,w_h)\tri \leq C\,\left(\|f\|_{-1,\Omega}+\|g_\sD\|_{1/2,\Gamma_D}
    +\|g_\sN\|_{-1/2,\Gamma_N}\right).
    \end{equation}
 \end{thm}

 Now, the finite element approximation to $(\bsigma,\,u)$ is defined as follows:
  \begin{equation}\label{s-u-h}
  \bsigma_h=\bfeta_h-A\,\grad\,w_h-\bb\,w_h
  \quad\mbox{and}\quad
  u_h=-a^{-1}\gradt\,\bfeta_h + w_h.
  \end{equation}
  
 \begin{rem}
  When the coefficients ($A$, $\bb$, and $a$) are not polynomials, they can be replaced 
  by their approximations of appropriate polynomials locally, if piecewise polynomial   
 approximation to $(\bsigma,\,u)$ is desirable.
  \end{rem}
  
  \begin{rem}
  The FOSLL* approximation to the solution $u$ is not continuous. To obtain a continuous 
  approximation, one can simply project $u_h$ onto appropriate continuous finite element space. 
  \end{rem}

  \begin{rem}
 In the case that $a=0$,   the finite element approximation to $(\bsigma,\,u)$ is given by
  \[
  \bsigma_h=\bfeta_h-A\,\grad\,w_h-\bb\,w_h
\quad\mbox{and}\quad
 u_h=-\gradt\,\bfeta_h .
\]
  \end{rem}

\section{A Priori Error Estimate}

Difference between equations in (\ref{fosll*}) and (\ref{fem}) gives the error equation:
\begin{equation}\label{ortho1}
b(\bfeta-\bfeta_h,\,w-w_h;\,\btau,\,v) = 0
\quad\forall\,\, (\btau,\, v) \in \Sigma^k_h\times V^{k+1}_h.
\end{equation}
The following error estimation in the energy norm is a simple
consequence of Theorem 3.1, the error equation in (\ref{ortho1}),
the Cauchy-Schwarz inequality, and the approximation properties in
(\ref{app2}) and (\ref{app1}).

\smallskip
\begin{thm}\label{th:app}
Assume that the solution $(\bsigma,\,u)$ of {\em (\ref{fos-s})-(\ref{bc-s})}
is in $H^{l}(\O)^{d}\times H^{l+1}(\O)$ and that the solution 
$(\bfeta,\,w)$ of {\em (\ref{fosll*})} 
satisfies
 \beq\label{assumption1}
 \|\gradt\bfeta\|_{l}+\|\bfeta\|_{l} +\|w\|_{l+1}
 \leq C\, \left(\|A^{1/2}\bsigma\|_{l}+ \|a^{1/2}u\|_{l}\right).
 \eeq
Let $k$ be the smallest integer
greater than or equal to $l-1$. Then the FOSLL* approximation $(\bsigma_h,\,u_h)$
defined in {\em (\ref{s-u-h})} has
the following error estimate 
\beq\label{err1a}
 \|A^{1/2}(\bsigma -\bsigma_h)\|+\|a^{1/2}(u-u_h)\|
 \leq C\,h^l
 \left(\|\bsigma\|_{l}+\|u\|_{l}\right)
 \leq C\,h^l \|u\|_{l+1}.
\eeq
\end{thm}

\begin{proof}
Let $(\bfeta_h,\,w_h)$ be the solution of (\ref{fem}). It follows from Theorem 2.1,
the error equation in (\ref{ortho1}), and the approximation properties in (\ref{app2}) and (\ref{app1}) that
\begin{eqnarray*}
\tri (\bfeta-\bfeta_h,\,w-w_h)\tri
&\leq & C\,\left(\inf_{\btau\in\,\Sigma^k_h}\|\bfeta -\btau\|_{\Hdiv}+
  \inf_{v\in V^{k+1}_h}\|w -v\|_1\right)\\ [2mm]
 &\leq & C\,h^l  \left(\|\gradt\bfeta\|_{l}+\|\bfeta\|_{l}+\|w\|_{l+1}\right),
\end{eqnarray*}
which, together with (\ref{assumption1}), implies (\ref{err1a}).
This completes the proof of the theorem.
\end{proof}

\begin{rem}
When $l=0$,  Assumption {\em (\ref{assumption1})} is the coercivity bound 
in {\em (\ref{coercivity})}.
\end{rem}

  \begin{rem}
In the case that $a=0$,   {\em (\ref{err1a})} becomes
\beq
 \|A^{1/2}(\bsigma -\bsigma_h)\|+\|(u-u_h)\|
 \leq C\,h^l
 \left(\|\bsigma\|_{l}+\|u\|_{l+1}\right)
 \leq C\,h^l \|u\|_{l+1}.
\eeq
  \end{rem}

\smallskip


\section{A Posteriori Error Estimate}

Unlike the FOSLS approach, the FOSLL* approach does not have a free a posteriori error estimator, thus in this section we study an explicit residual
error estimator and establish its reliability and efficiency bounds.


\subsection{Local Indicator and Global Estimator}

 Since the bilinear form $b(\cdot,\,\cdot)$ is coercive and continuous in 
 $H_N(\divvr;\O)\times H^1_D(\O)$ (see Theroem~\ref{co:equi}), the explicit residual a posteriori error estimator to be derived in this paper is a combination of those 
for the $H(\divvr)$
and the elliptic problems (see \cite{CNS:07, AiOd:00, Ver:96, Ver:13}).

To this end, we first introduce some notations. Denote by $\cE_\sK$ the set of edges/faces of element $K\in\cT_h$ and the set of edges/faces of the triangulation $\cT_h$ by
 $
 \cE_h := \cE_{_I}\cup\cE_{_D}\cup\cE_{_N}
 $,
 where $\cE_{_I}$ is the set of interior element edges, and $\cE_{_D}$ and $\cE_{_N}$ are the sets of boundary edges belonging to the respective $\Gamma_D$ and $\Gamma_N$. For each $e \in \cE$, denote by  $h_e$  the length/diameter of the edge/face $e$ and by $\bn_e$ a unit vector normal to $e$.  Let $K_e^{-}$ and $K_e^{+}$ be the two elements sharing the
common edge/face $e$ such that the unit outward normal vector of $K_e^{-}$ coincides with $\bn_e$. When $e \in \cE_{D} \cap \cE_{N}$, $\bn_{e}$ is the unit outward vector normal to $\p\O$ and denote by $K_e^{-}$ the element having the edge/face $e$. For a function $v$ defined on $K^{-}_e\cup K^{+}_e$, denote its traces on $F$ by $v|_e^{-}$ and $v|_e^{+}$, respectively. The jump over the edge/face $e$ is denoted by 
 $$
\jump{v}_e :=  \left\{
    \begin{array}{lll}
 	v|^{-}_e - v|_e^{+} \quad & e\in \cE_\sI,\\[2mm]
	v|^{-}_e  \quad & e \in \cE_{_D}\cup\cE_{_N}.
     \end{array}
\right.
$$
(When there is no ambiguity, the subscript or superscript $e$ in the designation of the jump  will be dropped.)  For a function $v$, we will use the following notations on the weighted 
$L^2$ norms:
 \[\begin{array} {rl}
  \|h\,v\|_{\cT_h} =\left( \sum\limits_{K\in\cT_h} \|h\,v\|_K^2\right)^{1/2}
 & \quad \mbox{where }\,\,
 \|h\,v\|_K= \|h_K\,v\|_{0,K} \quad\forall\,\, K\in\cT_h, \\[4mm]
\mbox{and}\quad
 \|h\,v\|_{\cE_h} =\left( \sum\limits_{e\in\cE_h} \|h\,v\|_e^2\right)^{1/2}
  & \quad \mbox{where }\,\,
  \|h\,v\|_e= \|h_e\,v\|_{0,e} \quad\forall\,\, e\in\cE_h.
 \end{array}\]

Let $(\bfeta_h,\,w_h)$ be the solution of (\ref{fem}) and let $(\bsigma_h,\,u_h)$ be the finite 
element approximation to $(\bsigma,\,u)$ defined in (\ref{s-u-h}). 
On each element $K\in\cT_h$, denote the following element residuals by
\[
r_1|_\sK=f -\gradt \bsigma_h+\bb\cdot A^{-1}\bsigma_h- a\,u_h,
\quad
\br_2|_\sK = A^{-1}\bsigma_h+\grad u_h,
\quad\mbox{and}\quad
r_3|_\sK= \curlt (A^{-1}\bsigma_h).
\]
Denote the following edge jumps by
  \[\begin{array}{llll}
   J_1|_e=\jump{\bsigma_h\cdot\bn},
 & J_2|_e  = \jump{u_h},
  &
   J_3|_e  = \jump{A^{-1}\bsigma_h\cdot\bt},
  &\mbox{ on } e\in\cE_I, \\[2mm]
  J_1|_e=0,
  &J_2|_e  = u_h-g_\sD,
  &
   J_3|_e  = \grad g_\sD\cdot\bt +A^{-1}\bsigma_h\cdot\bt,
  & \mbox{ on } e\in\cE_\sD, \\[2mm]
  J_1|_e=\bsigma_h\cdot\bn-g_\sN,
 & J_2|_e  =0,
  &
  J_3|_e  = 0,
 &\mbox{ on } e\in\cE_\sN.
\end{array}
\]
Let $\bar{\br}_2|_K$, $\bar{r}_1|_K$ and 
$\bar{r}_3|_K$, and $\bar{J}_i|_e$ ($i=1,\,2,\,3$) be the $L^2$-projections of  the respective
$\br_2|_K$, $r_1|_K$ and 
$r_3|_K$, and $J_i|_e$ ($i=1,\,2,\,3$) onto $P_k(K)^2$, $P_k(K)$, and $P_k(e)$, respectively. 
Now, the local error indicator on each element $K\in \cT_h$ is defined by
\begin{eqnarray}\nonumber
\eta_\sK^2 
 &=&   \|h\, \bar{r}_1\|_{K}^2 + \|h\, \bar{\br}_2\|_{K}^2 + \|h\, \bar{r}_3\|_{K}^2 
 +\dfrac{1}{2}\sum_{e\in\cE_\sI \cap  \cE_\sK}\left( \|h^{1/2} \bar{J}_1\|_{e}^2
 + \|h^{1/2} \bar{J}_3\|_{e}^2
 \right)
 \\[2mm]\label{indicator}
&&
 +\sum_{e\in\cE_\sD \cap  \cE_\sK} \|h^{1/2} \bar{J}_3\|_{e}^2
+\sum_{e\in\cE_\sN \cap \cE_\sK} \|h^{1/2} \bar{J}_1\|_{e}^2,
\end{eqnarray}
and the global error estimator is defined by 
 \beq\label{estimator}   
\eta^2 =\sum_{K\in\cT_h}\eta_\sK^2 
= \|h\,\bar{r}_1\|_{\cT_h}^2 + \|h\,\bar{\br}_2\|_{\cT_h}^2 
 + \|h\,\bar{r}_3\|_{\cT_h}^2
  +\|h^{1/2} \bar{J}_1\|_{\cE_h}^2 
+ \|h^{1/2} \bar{J}_3\|_{\cE_h}^2. 
  \eeq

The terms $r_1$ and $\br_2$ are the residuals of the equations in (\ref{fos-s}). The term 
$r_3$ measures the
violation of the fact that the exact quantity $-A^{-1}\bsigma = \grad u$ is in the kernel of $\curlt$ operator. The terms $J_1$, $J_2$, and $J_3$ are  due to the facts that the numerical 
flux $\bsigma_h$, the numerical solution $u_h$, and the numerical gradient 
$-A^{-1}\bsigma_h$
are not in $H(\divvr;\Omega)$, $H^1(\Omega)$, and $H(\curll;\Omega)$, respectively.



\subsection{Reliability and Efficiency Bounds}

For simplicity, we analyze only two dimensions here since there is no essential difficulties for 
three dimensions.
For a vector field $\btau=(\tau_1, \tau_2)^t$ and a scalar-value function $v$, define 
the respective curl operator and its formal adjoint by 
 \[
 \curlt \btau := \dfrac{\p \tau_2}{\p x_1} -\dfrac{\p \tau_1}{\p x_2}
\quad\mbox{and}\quad
 \gperp v:= (-\dfrac{\p v}{\p x_1}, \dfrac{\p v}{\p x_1})^t.
\]

Denote by $\Pi_h: H_N(\divvr;\O)\cap L^s(\O)^2 \rightarrow \Sigma^0_h$ with $s>2$ the $RT_0$ interpolation operator; i.e., for all $\btau \in H^1(\O)^2$, one has \cite{BoBrFo:13}
\begin{eqnarray} \label{propRT}
&& \langle  (\btau - \Pi_h \btau)\cdot\bn, v\rangle_e = 0, \quad \forall\,\, v\in P_0(e) \quad\mbox{and}\quad \forall\,\, e\in \cE.
\end{eqnarray}
Let $S_h$ be the standard continuous piecewise linear finite element space on the triangulation 
$\cT_h$. For $B=D$ or $N$, denote by 
$I_{\sB}: H_B^1(\O)\rightarrow H_B^1(\O)\bigcap S_h$ the Clement interpolation operator
which satisfies the following local approximation property \cite{Bra:07}:
$$
\|h^{-1}(v - I_{\sB} v)\|_{\cT_h}  \leq C\, \|\grad v\|  \quad \forall\,\, v\in H^1_B(\O) ,
$$
where $B=D$ or $N$. It is easy to check that  $\gperp (I_{\sN} v) \in \Sigma_h^0 $.

Let $(\bfeta,\,w)$ be the solution of (\ref{fosll*}), and let 
 \beq\label{error}
  \bE=\bfeta-\bfeta_h\in H_N(\divvr;\O)
 \quad\mbox{and}\quad
 e=w-w_h \in H^1_D(\O).
\eeq
By Lemmas 5.1 in \cite{CNS:07}, the $\bE$ has the following quasi-Helmholtz decomposition:
 \beq\label{QHD}
  \bE = \bphi +\gperp \psi \quad\mbox{in }\,\, \O,
 \eeq
where $\bphi\in H^1_{N}(\O)^2$ and $\psi \in H_N^1(\O)$. Moreover, there exists a 
constant $C>0$ such that
 \beq\label{QHD2}
 \|\nabla \bphi\|  \leq C\,\tri\bE\tri_{H(\divvr)} 
 \quad\mbox{and}\quad
 \|\nabla \psi\|  \leq C\,\tri \bE\tri_{H(\divvr)}.
 \eeq
Let 
 \[
 \bphi _h := \Pi_h \bphi,\quad
 \psi _h :=  I_{\sN} \psi,\quad
\mbox{and}\quad
e_h:=I_\sD e
\]
and let 
 \[
 \tilde{e}=e-e_h,\quad
 \tilde{\bphi}=\bphi-\bphi_h,\quad
 \mbox{and}\quad
\tilde{\psi}=\psi-\psi_h.
\]
By the approximation properties of the interpolation operators and (\ref{QHD2}), we have
\begin{eqnarray*}
 \|h^{-1} \tilde{e}\|_{\cT_h}+\|h^{-1/2} \tilde{e} \|_{\cE_h}
 &\leq & C\,\|\grad \,e\|   \leq C\, \tri e\tri_1,\\[2mm]
\|h^{-1} \tilde{\bphi}\|_{\cT_h} + \|h^{-1/2} \tilde{\bphi}\|_{\cE_h}  
 &\leq & C\,\|\grad \bphi\|   \leq C\, \tri \bE\tri_{H(\divvr)},\\[2mm]
 \mbox{and}\quad
 \|h^{-1} \tilde{\psi}\|_{\cT_h}  + \|h^{-1/2} \tilde{\psi}\|_{\cE_h}
 &\leq & C\,\|\grad \psi\|   \leq C\, \tri \bE\tri_{H(\divvr)}. 
\end{eqnarray*}

\begin{lem} \label{rep}
We have the following error representation:
 \beq \label{errorrep}
b(\bE,e; \bE, e) 
 = \sum_{K\in\cT_h} \big\{ (r_1,\,\tilde{e})_\sK - (\br_2,\, \tilde{\bphi})_\sK 
 -  (r_3, \,\tilde{\psi})_\sK\big\} 
 + \sum_{e\in\cE_h} \big\{ \langle J_1, \,\tilde{e}\rangle_e 
+ \langle J_2, \tilde{\bphi}\cdot\bn\rangle_e+ \langle J_3, \tilde{\psi}\rangle_e\big\}.
\eeq
\end{lem}

\begin{proof}
Since  $\bE_h:= \bphi _h +\gperp \psi _h \in \Sigma_h^0\subset \Sigma_h^k$
and $e_h\in V^1_h\subset V^{k+1}_h$, the error equation in (\ref{ortho1}) gives
 \beq\label{ortho2}
 b(\bE,e;\,\bE,e) 
 = b(\bE,e;\,\bE-\bE_h,e-e_h)
 =b(\bE,e;\,\tilde{\bphi} +\gperp \tilde{\psi},\tilde{e}).
\eeq
By the fact that
$\gradt\gperp \tilde{\psi}=0$, the definitions of
the bilinear form $b(\cdot,\,\cdot)$, and the FOSLL* finite element approximation 
$(\bsigma_h,\,u_h)$, we have
 \begin{eqnarray*}
 && b(\bfeta_h, w_h;\,\tilde{\bphi} +\gperp \tilde{\psi},\tilde{e})
 = (\bsigma_h,\,  A^{-1}(\tilde{\bphi} +\gperp \tilde{\psi}) 
  - \grad \tilde{e}- A^{-1}\bb\, \tilde{e})
 -(u_h, \gradt \tilde{\bphi} - a\,\tilde{e})\\[2mm]
 &&\qquad  =( A^{-1}\bsigma_h,\,\tilde{\bphi} +\gperp \tilde{\psi})
-(u_h,\,\gradt\tilde{\bphi})
      -( \bsigma_h,\,\grad \tilde{e}) -(\bb^tA^{-1}\bsigma_h-a\,u_h,\,\tilde{e}).
\end{eqnarray*}
It follows from integration by parts and the boundary conditions that
\begin{eqnarray*}
 ( A^{-1}\bsigma_h,\, \gperp \tilde{\psi})
 &=& \!\!\sum_{K\in \cT_h} \big(\grad\times(A^{-1}\bsigma_h),\, \tilde{\psi})_K
   -\!\!\sum_{e\in \cE_\sI} \langle \jump{A^{-1}\bsigma_h\cdot\bt},\, \tilde{\psi} \rangle_e
  -\!\!\sum_{e\in \cE_D} \langle A^{-1}\bsigma_h\cdot\bt,\, \tilde{\psi} \rangle_e,\\[2mm]
(u_h,\,\gradt\tilde{\bphi})
 &=& -\sum_{K\in\cT_h}(\grad u_h,\, \tilde{\bphi})_K 
  + \sum_{e\in \cE_\sI} \langle \jump{u_h}, \,\tilde{\bphi}\cdot\bn \rangle_e
  + \sum_{e\in \cE_D} \langle u_h, \,\tilde{\bphi}\cdot\bn \rangle_e,\\ [2mm]
  ( \bsigma_h,\,\grad \tilde{e}) 
 &=& -\sum_{K\in \cT_h} (\gradt \bsigma_h,\, \tilde{e})_K 
  + \sum_{e\in \cE_\sI} \langle \jump{\bsigma_h\cdot\bn},\, \tilde{e}\rangle_e 
  + \sum_{e\in \cE_N}\langle \bsigma_h\cdot\bn,\, \tilde{e}\rangle_e, 
\end{eqnarray*}
which, together with the definitions of the residuals and the jumps, lead to
 \begin{eqnarray*}
   && b(\bfeta_h, w_h; \tilde{\bphi}+\gperp \tilde{\psi}, \tilde{e}) \\[2mm]
  &=&\sum_{K\in \cT_h} \left\{\big(r_3,\,\tilde{\psi}\big)_K
    + \big( \br_2,\,\tilde{\bphi}\big)_K - \big( r_1,\,\tilde{e}\big)_K\right\}
  +\sum_{K\in \cT_h} \big( f,\,\tilde{e}\big)_K
  -\sum_{e\in \cE_I} \langle J_3,\, \tilde{\psi} \rangle_e
 \\[2mm]
 &-& \!\!\sum_{e\in \cE_D} \langle A^{-1}\bsigma_h\cdot\bt,\, \tilde{\psi} \rangle_e
   -\sum_{e\in \cE} \langle J_2,\, \tilde{\bphi}\cdot\bn\rangle_e
  -\sum_{e\in \cE} \langle J_1,\, \tilde{e} \rangle_e
  - \sum_{e\in \cE_\sD} \langle g_\sD,\, \tilde{\bphi}\cdot\bn \rangle_e
 -\sum_{e\in \cE_\sN} \langle g_\sN,\, \tilde{e} \rangle_e.
\end{eqnarray*}
By (\ref{fosll*}), integration by parts, 
and the boundary condition of $\tilde{\psi}\in H^1_N(\O)$, we have
 \begin{eqnarray*}
  b(\bfeta, w; \tilde{\bphi}+\gperp \tilde{\psi}, \tilde{e})
 &=&(f,\,\tilde{e})-\langle g_\sN, \tilde{e}\rangle _{\Gamma_N} 
 - \langle g_\sD, (\tilde{\bphi} + \gperp \tilde{\psi}) \cdot \bn\rangle _{\Gamma_D}\\[2mm]
 &=&(f,\,\tilde{e})-\langle g_\sN, \tilde{e}\rangle _{\Gamma_N} 
 - \langle g_\sD, \,\tilde{\bphi} \cdot \bn\rangle _{\Gamma_D}
 +\langle \grad g_\sD\cdot \bt, \,\tilde{\psi}\rangle _{\Gamma_D}.
\end{eqnarray*}
Now, (\ref{errorrep}) is a direct consequence of (\ref{ortho2}) and
the difference of the above two equalities. This completes the
proof of the lemma.
\end{proof}

Define the local and global oscillations as follows:
\begin{eqnarray}
\nonumber
\osc_k^2(K) &=& \|h\, (r_1-\bar{r}_1)\|_{K}^2 + \|h\, (\br_2-\bar{\br}_2)\|_{K}^2 + \|h\, (r_3-\bar{r}_3)\|_{K}^2 +\|h^{1/2} (J_1-\bar{J}_1)\|_{\p K}^2\\[2mm]
&& + \|h^{1/2} (J_2-\bar{J}_2)\|_{\p K}^2  + \|h^{1/2} (J_3-\bar{J}_3)\|_{\p K}^2\\[2mm]
\mbox{and}\quad \osc_k^2(\cT_h) &=& \sum_{K\in\cT_h} \osc_k^2(K),
\end{eqnarray}
respectively. 

\begin{thm} {\em (Reliability Bound)} The global estimator $\eta$ defined in {\em (\ref{estimator})} is reliable; i.e., there exists a positive constant $C$ such that 
\beq\label{reliability}
  \|A^{-1/2}(\bsigma-\bsigma_h)\|  +\|a^{1/2}(u-u_h)\|  
  \leq C\, \tri (\bE, e)\tri  
 \leq C\, (\eta +\osc_k(\cT_h)).
\eeq
\end{thm}

\begin{proof}
The first inequality in (\ref{reliability}) is a direct consequence of 
the definition of $(\bsigma_h,\,u_h)$ and the triangle inequality. 

To show the validity of the second inequality in 
(\ref{reliability}), by the coercivity in (\ref{coercivity}), it suffices to show that
 \beq\label{5.19}
  b(\bE,e;\,\bE,e) 
 \leq C\,(\eta +\osc_k(\cT_h))\, \tri (\bE,e)\tri.
 \eeq
To this end, first  notice that by the property in (\ref{propRT}) and the definition of $\bar{J}_2$, we have 
 \[
 \langle J_2, (\bphi-\bphi_h)\cdot\bn\rangle_e = \langle J_2-\bar{J}_2,  (\bphi-\bphi_h)\cdot\bn\rangle_e. 
\]
Now, it follows from Lemma~\ref{rep}, the Cauchy-Schwarz inequality, and the approximation properties of $\tilde{e}$, $\tilde{\bphi}$, and $\tilde{\psi}$, and the triangle inequality that
\begin{eqnarray*}
  && b(\bE,e;\,\bE,e)\\[2mm]
 &=&
  \sum_{K\in\cT_h} \big( (r_1,\,\tilde{e})_\sK + (\br_2,\, \tilde{\bphi})_\sK 
 +  (r_3, \,\tilde{\psi})_\sK\big) 
 + \sum_{e\in\cE} \big( \langle J_1, \,\tilde{e}\rangle_e 
 + \langle J_2-\bar{J}_2, \tilde{\bphi}\cdot\bn\rangle_e+ \langle J_3, \tilde{\psi}\rangle_e\big)
  \\[2mm]
&\leq& \sum_{K\in\cT_h} \big(\|h\,r_1\|_{K}\,\|h^{-1} \tilde{e} \|_{K} 
 + \|h\,\br_2\|_{K}\,\|h^{-1}\tilde{\bphi}\|_{K}
 + \|h\, r_3\|_{K}\,\|h^{-1}\tilde{\psi}\|_{K} \big) \\[2mm]
  &+&  \sum_{e\in\cE_h} \left( \|h^{1/2}J_1\|_{e}\,\|h^{-1/2}\tilde{e}\|_{e} 
 +  \|h^{1/2} (J_2-\bar{J}_2)\|_{e}\,\|h^{-1/2} \tilde{\bphi}\cdot\bn\|_{e}
+ \|h^{1/2} J_3\|_{e}\,\|h^{-1/2}\tilde{\psi}\|_{e} \right) \\[2mm]
 &\leq& C\,\left(\sum_{i=1, 3}\big(\|h\,r_i\|_{\cT_h}^2+\|h^{1/2}J_i\|_{\cE_h}^2\big)
  + \|h\,\br_2\|_{\cT_h}^2
 +  \|h^{1/2} (J_2-\bar{J}_2)\|_{\cE_h}^2
 \right)^{1/2} \, \tri (\bE,e)\tri \\[2mm]
&\leq& C\,(\eta + \osc_k(\cT_h))\, \tri (\bE,e)\tri,
\end{eqnarray*}
which proves (\ref{5.19}) and, hence, the theorem.
\end{proof}

%

\begin{thm}\label{eff}
 {\em (Local Efficiency Bound)}
For all $K\in\cT_h$, the local error indicator $\eta_K$ defined in {\em (\ref{indicator})} is  efficient; i.e., there exists a positive constant $C$ such that 
\beq\label{efficiency}
C\,\eta_K \leq \|\bsigma-\bsigma_h\|_{\o_K} + \|u-u_h\|_{\o_K}+ \osc_k(\o_K),
\eeq
where $\o_K$ is the union of elements in $\cT_h$ sharing an edge with $K$.
\end{thm}

The proof of the local efficiency bound in Theorem \ref{eff} is standard; i.e., it is proved 
by using local edge and element bubble functions, $\phi_e$ and $\phi_K$ (see \cite{Ver:96} for their definitions and properties). For simplicity, we only sketch the proof below.

\begin{proof} 
For any $(\btau,\,v)\in H_N(\divvr;\O)\times H_D^1(\O)$, by the quasi-Helmholtz 
decomposition, we have
 \[ 
 \btau =\bphi +\gperp \psi \in H_N(\divvr;\O),
 \]
where $\bphi \in H_N^1(\O)$ and $\psi \in H_N^1(\O)$.
The same argument as the proof of Lemma~\ref{rep} gives
$$
b(\bE,e; \btau, v) = \sum_{K\in\cT_h} \left\{ (r_1, v)_\sK + (\br_2, \bphi)_\sK +  (r_3, \psi)_\sK\right\}
+ \sum_{e\in\cE_h} \left\{\langle J_1, v\rangle_e + \langle J_2, \bphi\cdot\bn\rangle_e+ \langle J_3, \psi\rangle_e\right\},
$$
which, together with the definitions of $\bsigma$ and $u$, yields
$$
b(\bE,e; \btau, v) =
\left(\bsigma-\bsigma_h,\,A^{-1}\btau - (\grad+A^{-1}\bb)\,v \right) 
   - (u-u_h,\, \gradt\,\btau-a\, v).
$$
Hence,
\begin{eqnarray}\nonumber
&&    \sum_{K\in\cT_h} \left( (r_1, v)_\sK + (\br_2, \bphi)_\sK +  (r_3, \psi)_\sK\right)
+ \sum_{e\in\cE_h} \left( \langle J_1, v\rangle_e + \langle J_2, \bphi\cdot\bn\rangle_e+ \langle J_3, \psi\rangle_e\right)\\[2mm]
\label{errrep}
&=&\left(\bsigma-\bsigma_h,\,A^{-1}(\bphi +\gperp \psi) - (\grad+A^{-1}\bb)\,v \right) 
   - (u-u_h,\, \gradt\,\bphi-a\, v).
\end{eqnarray}
In (\ref{errrep}), by choosing (1) $\bphi=0$, $\psi=0$, and $v= \bar{r}_1\,\phi_K$;
(2) $\bphi= \bar{\br}_2\,\phi_K$, $\psi=0$, and $v=0$; 
and (3) $\bphi=0$, $\psi=  \bar{r}_3\,\phi_K$,  and $v=0$
and by the standard argument, we can then establish
upper bounds for the element residuals, 
$\|h\,\bar{r}_1\|_K$, $\|h\,\bar{\br}_2\|_K$, and $\|h\,\bar{r}_3\|_K$, respectively.
In a similar fashion, 
to bound the edge jumps $\|h\,\bar{J}_1\|_e$ and $\|h\,\bar{J}_3\|_e$ above,
we choose (1) $\bphi=0$, $\psi=0$, and $v = \bar{J}_1\,\phi_e$
and (2) $\bphi=0$, $\psi= \bar{J}_3\,\phi_e$, and $v=0$ in (\ref{errrep}), respectively.
\end{proof}


\section{Numerical Results}

In this section, numerical results for a second order elliptic partial differential 
equation are presented. 

We begin with discretizations of a test problem on a sequence of uniform
meshes to verify the a priori error estimation. The test problem is defined on 
$\O = (0,1)^2$ with coefficients $A = I$, $\bb = (3,2)^t$, and $a=2$. The exact 
solution of this problem is $u=\sin(\pi x)\sin(\pi y)$ 
with homogeneous boundary condition on $\p \O$. 
Finite element spaces $\Sigma_h^0$ 
and $V_h^1$ are used to approximate $\btau$ and $w$, respectively. 
Table 6.1 shows that the convergence rates of the errors for the original variables $\bsigma$ and $u$ in $L^2$-norms are optimal.

\begin{table}[htdp]
\caption{Errors $\|\bsigma-\bsigma_h\| $, $\|u-u_h\| $, $\|\bsigma-\bsigma_h\| +\|u-u_h\| $, and convergence rates }
\begin{center}
\begin{tabular}{|c|c|c|c|c|c|c|}
\hline
h & $\|\bsigma-\bsigma_h\| $ & rate & $\|u-u_h\| $ & rate & $\|\bsigma-\bsigma_h\| +\|u-u_h\| $& rate\\
\hline
1/8 & 5.859E-1 &&     4.351E-2 && 6.294E-1& \\
\hline
1/16& 2.972E-1 & 1.971 & 1.601E-2 & 2.718&3.132E-1&2.010 \\
\hline
1/32&1.492E-1&1.992& 7.013E-3&2.283&1.562E-1&2.005 \\
\hline
1/64&7.466E-2&1.998& 3.367E-3&2.013&7.802E-2 & 2.002\\
\hline
1/128&3.734E-2&2.000&1.666E-3& 2.201&3.900E-2 &2.006\\
\hline
\end{tabular}
\end{center}
\label{default}
\end{table}%

The next example is to test the a posteriori error estimator. 
The test problem is the Laplace equation $-\Delta u =0$ defined on
an L-shaped domain $\O := (-1, 1)^2 \backslash {[0, 1)\times(-1, 0]}$
with a reentrant corner at the origin. The Dirichlet boundary condition $u|_{\p\O}=g_\sD$
is chosen such that the exact solution is
$
 u(r,\theta) = r^{2/3} \sin(2\theta)
$
in polar coordinates. 
Starting with the coarsest triangulation 
$\cT_0$ obtained from halving $12$ uniform squares, a sequence of meshes
is generated by using the standard adaptive meshing algorithm that
adopts the bulk marking strategy with bulk marking parameter $0.5$.
Marked triangles are refined by bisection.

\begin{figure}[!hts]
    \hfill
    \begin{minipage}[!hbp]{0.48\linewidth}
        \centering
        \includegraphics[width=0.99\textwidth,angle=0]{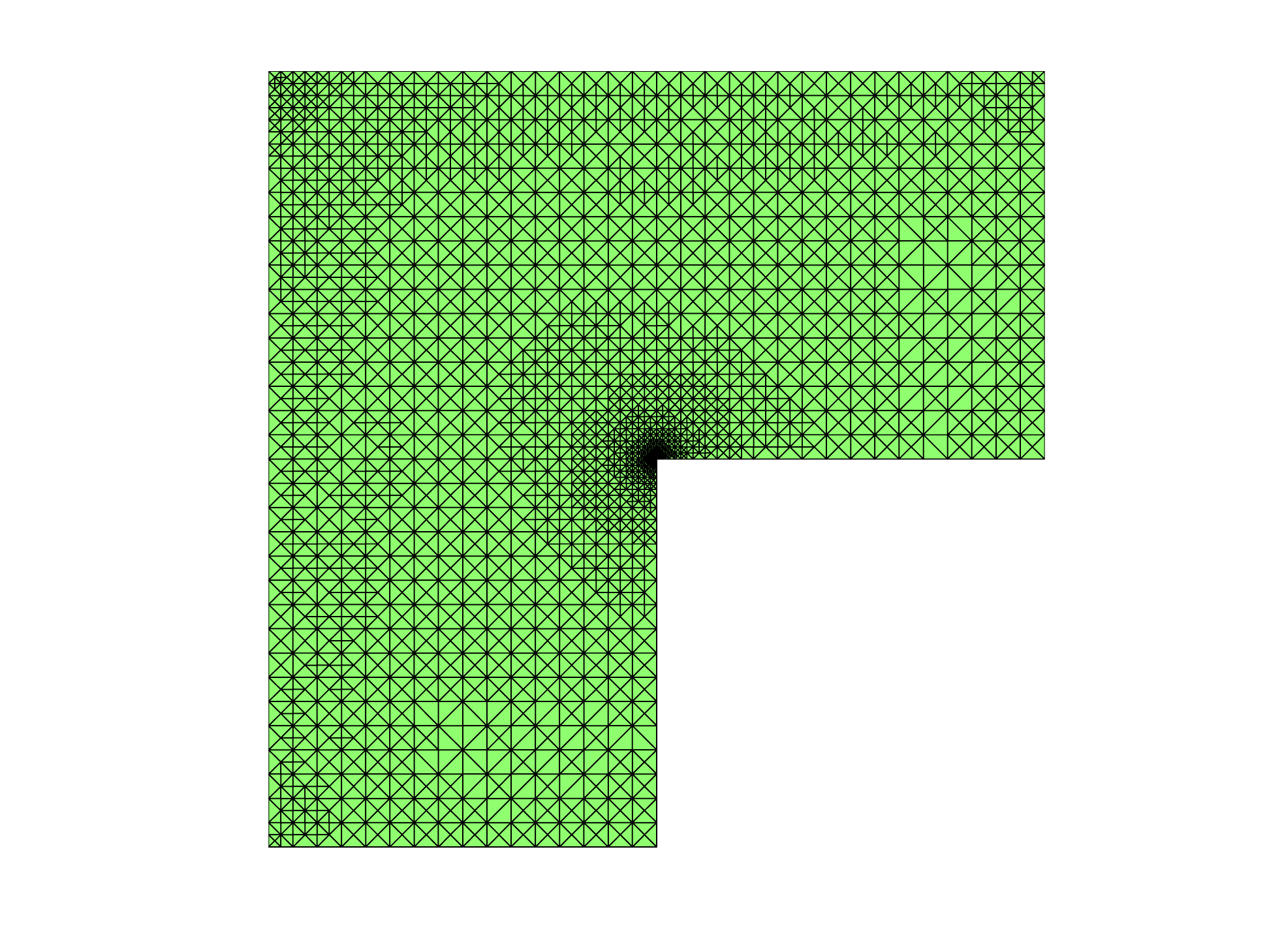}
        \caption{mesh generated by error estimator}%
        \end{minipage}%
        \quad
    \begin{minipage}[!htbp]{0.48\linewidth}
        \centering
        \includegraphics[width=0.99\textwidth,angle=0]{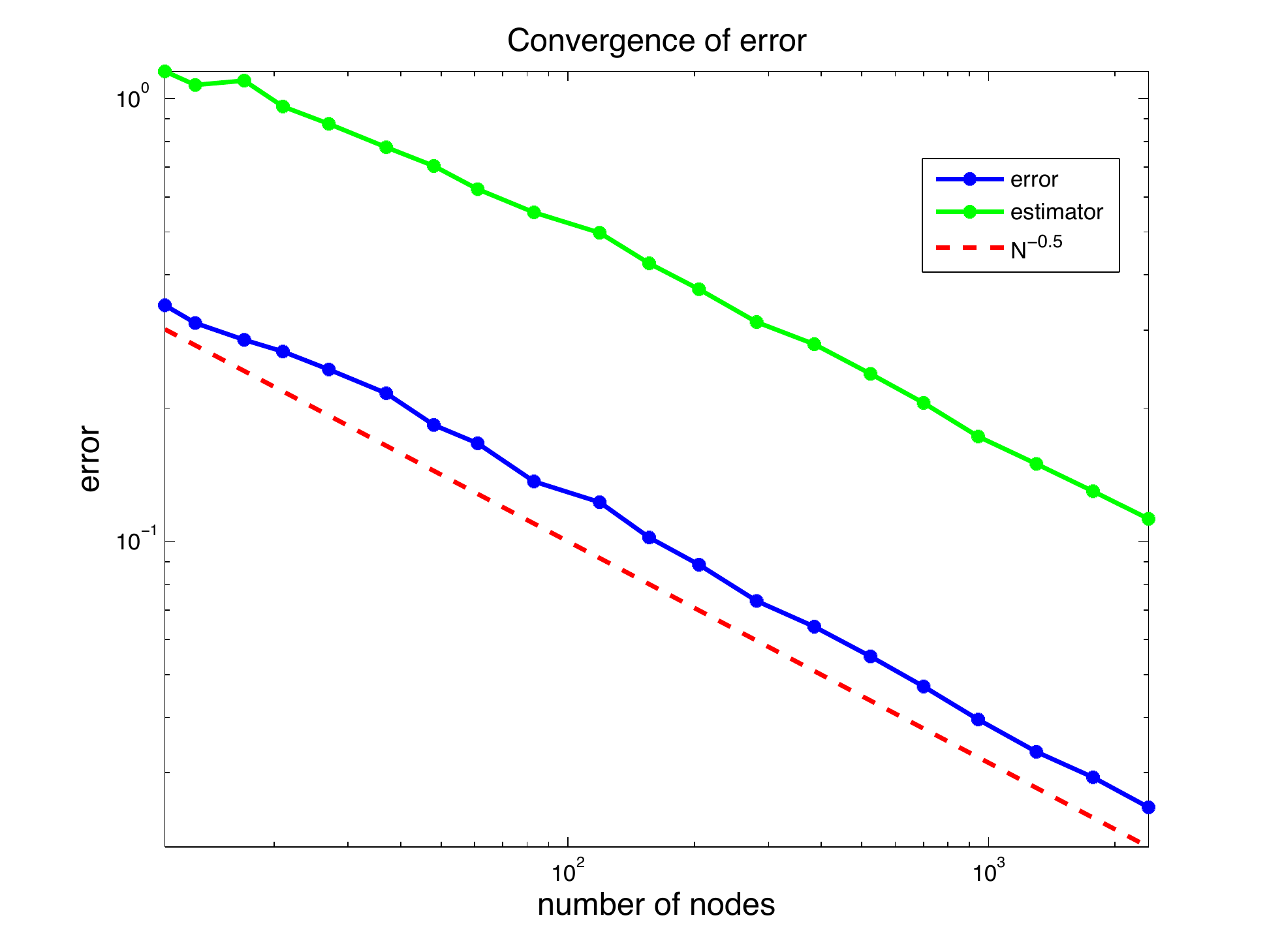}
        \caption{error $\|\bsigma-\bsigma_h\| +\|u-u_h\| $ and estimator}%
    \end{minipage}%
        \hfill
\end{figure}

Mesh generated by $\eta$ is shown in Figure 6.1. 
The refinement is mainly around the reentrant corner.
Comparison of the true error and the $\eta$ is shown in Figure 6.2.
Moreover, the slope of the log(dof)- log(error) 
for the $\eta$ and the true error is $-1/2$, which indicates the optimal decay of 
the error with respect to the number of unknowns.

\bigskip
\bibliographystyle{siam}

\begin{thebibliography}{10}
 \bibitem{AiOd:00}
 {\sc M. Ainsworth and J. T. Oden},
 {\em A Posteriori Error Estimation in Finite Element Analysis},
 John Wiley \& Sons, Inc., 2000.


\bibitem{aks}
{\sc A. Aziz, R. Kellogg, and A. Stephens}, {\em Least-squares
methods for elliptic systems}, Math. Comp., 44 (1985), 53-70.
%

\bibitem{BocGun:09}
{\sc P.~B. Bochev and M.~D. Gunzburger}, {\em Least Squares Finite Element Methods}, 
Springer, Berlin, 2009

 \bibitem{BoGu:05}
 {\sc P. B. Bochev and M. D. Gunzburger},
 {\em On least-squares finite element methods for the Poisson equation
 and their connection to the Dirichlet and Kelvin principles},
 SIAM J. Numer. Anal., 43:1 (2005), 340-362.

\bibitem{BoBrFo:13}
  {\sc D. Boffi, F. Brezzi, and M. Fortin}, 
  {Mixed Finite Element Methods and Applications},
Springer Series in Computational Mathematics, 44, Springer, 2013.

\bibitem{Bra:07}
{\sc D. Braess},
{\em Finite Elements: Theory, Fast Solvers and Applications in Solid Mechanics},
3rd edition, Cambridge University Press, Cambridge, UK, 2007.

\bibitem{blp}
{\sc J. H. Bramble, R. D. Lazarov, and J. E. Pasciak}, {\em A
least-squares approach based on a discrete minus one inner product
for first order system}, Math. Comp., 66 (1997), 935-955.



\bibitem{clmm}
{\sc Z.~Cai, R. Lazarov, T.A. Manteuffel and S.F. McCormick}, {\em
First-order system least squares for second-order partial
differential equations: part I.}, SIAM J. Numer. Anal., 31 (1994),
1785-1799.

\bibitem{CaMaMc:97}
{\sc Z.~Cai,  T.A. Manteuffel and S.F. McCormick},
{\em  First-order system least squares for second-order partial differential equations: Part II},
SIAM J. Numer. Anal., 34:2 (1997), 425--454.

\bibitem{CaMaMcRu:01}
{\sc Z.~Cai, T.A. Manteuffel, S.F. McCormick, and J. Ruge}, {\em
First-order system LL* (FOSLL*) for scalar partial
differential equations}, SIAM J. Numer. Anal., 39 (2001),
1418-1445.

\bibitem{CaKu:10}
{\sc Z.~Cai and Ku}, 
{\em Optimal error estimates for the div least-squares method with data f in L2 and application to nonlinear problems},
SIAM J. Numer. Anal., 47:6 (2010), 4098-4111.


%




\bibitem{CNS:07}
{\sc J. M. Cascon, R. H. Nochetto, and K. G. Siebert},
{\em Design and convergence of afem in H(div)},
Math. Models and Methods in Applied Sciences, 17(2007), 1849-1881.



\bibitem{Cia:78}
{\sc P. G. Ciarlet}, {\em The Finite Element Method for Elliptic
Problems}, North-Holland, Amsterdam, 1978.



 \bibitem{Jes:77}
 {\sc D. C. Jespersen},
 {\em A least-square decomposition method for solving elliptic systems},
 Math.~Comp., 31 (1977), 873-880.

\bibitem{Jia:98}
{\sc B. N. Jiang},
{\em The Least-Squares Finite Element Method: Theory and Applications
in Computational Fluid Dynamics and Electromagnectics},
Spring, Berlin, 1998.





 \bibitem{PeCaLa:94}
 {\sc A. I. Pehlivanov, G. F. Carey, and R. D. Lazarov},
 {\em Least squares mixed finite elements for second order
      elliptic problems},
 SIAM J. Numer. Anal., 31 (1994), 1368-1377.


\bibitem{rt}
{\sc P. A. Raviart and J. M. Thomas}, {\it A mixed finite element
method for 2-nd order elliptic problems}, Mathmatical Aspects of
Finite Element Methods, Lecture Notes in Mathematics, \#606, I.
Galligani and E. Magenes, eds., Springer-Verlag, New York, 1977,
292-315.




 \bibitem{Ver:96}
 {\sc R. Verf\"{u}rth},
 {\em A Review of A-Posteriori Error Estimation and Adaptive Mesh
 Refinement Techniques},
 John Wiley and Teubner Series, Advances in Numerical Mathematics, 1996.
 
  \bibitem{Ver:13}
 {\sc R. Verf\"{u}rth},
 {\em A Posteriori Error Estimation Techniques for Finite Element Methods},
Oxford Numerical Mathematics and Scientific Computation, Oxford University Press, 2013.

\end{thebibliography}

\end{document}